\documentclass[psamsfonts]{amsart}

\usepackage{amssymb}
\usepackage{amsmath}
\usepackage{amsthm}
\usepackage[urlcolor=black, colorlinks=true, citecolor=black, linkcolor=black]{hyperref}
\usepackage[english]{babel}
\usepackage{enumerate}

\usepackage{dsfont} %Indicator function 1

\newcommand{ \R } { \mathbb{R} }
\newcommand{ \N } { \mathbb{N} }

\newcommand{\Q} {\mathbb{Q}}

\newcommand{\w}{\omega}
\newcommand{\wstar}{\omega^*}
\newcommand{\Cstar}{C^*}
\newcommand{\cont}{\mathfrak{c}}

\newcommand{\script}{\mathcal}
\newcommand{\parentheses}[1]{{\left( {#1} \right)}}
\newcommand{\p}{\parentheses}
\newcommand{\of}{\parentheses}

\newcommand{\singletonDeletion}[1]{\script{D}\parentheses{#1}}
\newcommand{\closure}[1]{\overline{#1}}

\newcommand{\interior}[1]{\mathrm{int}\of{#1}}

\newcommand{\Set}[1]{{\left\lbrace {#1} \right\rbrace}}
\newcommand{\singleton}{\Set}
\newcommand{\union}{\cup}

\newcommand{\Union}{\bigcup}

\newcommand{\cardinality}[1]{{\left\lvert {#1} \right\rvert}}

\def\set#1:#2{\Set{{#1} \colon {#2}}}

\begin{document}
\title{The space $\wstar$ and the reconstruction of normality}
\author{Max F.\ Pitz}
\address{Mathematical Institute\\University of Oxford\\Oxford OX2 6GG\\United Kingdom}
\email{pitz@maths.ox.ac.uk}
%\phone{+44 755 2202212}

\begin{abstract}
The topological reconstruction problem asks how much information about a topological space can be recovered from its point-complement subspaces. If the whole space can be recovered in this way, it is called reconstructible. 

Our main result states that it is independent of the axioms of set theory (ZFC) whether the Stone-\v{C}ech remainder of the integers $\wstar$ is reconstructible. Our second result is about the reconstruction of normality. We show that assuming the Continuum Hypothesis, the compact Hausdorff space $\wstar$ has a non-normal reconstruction, namely the space $\wstar \setminus \singleton{p}$ for a $P$-point $p$ of $\wstar$. More generally, we show that the existence of an uncountable cardinal $\kappa$ satisfying $\kappa = \kappa^{<\kappa}$ implies that there is a normal space with a non-normal reconstruction. 

These results demonstrate that consistently, the property of being a normal space is not reconstructible. Whether normality is non-reconstructible in ZFC remains an open question.
\end{abstract}

\maketitle
\thispagestyle{plain}

\newtheorem{myclm}{Claim}
\newtheorem*{recproblem}{The Topological Reconstruction Problem}
\newtheorem*{myclmn}{Claim}
\newtheorem{mycase}{Case}
\newtheorem{mythm}{Theorem}\numberwithin{mythm}{section} 
\newtheorem{myprop}[mythm]{Proposition}
\newtheorem{mycor}[mythm]{Corollary}
\newtheorem{mylem}[mythm]{Lemma} 
\newtheorem{myquest}{Question}

\section{The topological reconstruction problem}

In 1941, Ulam and Kelly proposed the \emph{reconstruction conjecture} in graph theory. This conjecture asks whether every finite graph with at least three vertices is uniquely determined by its unlabelled subgraphs obtained by deleting a single vertex and all incident edges. To this day, the reconstruction conjecture has remained open, and is considered one of the most challenging problems in graph theory. For a survey about the reconstruction conjecture see for example \cite{Bondy}.

The present paper is concerned with the topological version of the reconstruction conjecture, introduced in \cite{recpaper}. A topological space $Y$ is called a \emph{card} of another space $X$ if $Y$ is homeomorphic to $X \setminus \singleton{x}$ for some $x$ in $X$. The \emph{deck} $\singletonDeletion{X}$ of a space $X$ is a transversal for the non-homeomorphic cards of $X$, i.e.\ a set recording the topologically distinct subspaces one can obtain by deleting singletons from $X$. For example, if $Y$ is a card of the real line, then $Y$ is homeomorphic to two copies of the real line. We write this as $\singletonDeletion{\R}=\Set{\R \oplus \R}$. For $n$-dimensional Euclidean spaces we have $\singletonDeletion{\R^n}=\Set{\R^n \setminus \singleton{0}}$. For the unit interval $I$ we have $\singletonDeletion{I}=\Set{[0,1), [0,1) \oplus [0,1)}$. The Cantor set has the deck $\singletonDeletion{C}=\singleton{C \setminus \singleton{0}}$ and the rationals $\Q$ and irrationals $P$ have decks $\singletonDeletion{\Q} = \Set{\Q}$ and $\singletonDeletion{P}=\Set{P}$.
 
Two topological spaces are said to be \emph{reconstructions} of each other if their decks are indistinguishable. A space $X$ is said to be \emph{reconstructible} if all its reconstructions are homeomorphic to $X$. In the same spirit, we say that a property of topological spaces is reconstructible if it is preserved under reconstruction. Formally, a space $X$ is reconstructible if $\singletonDeletion{X}=\singletonDeletion{Z}$ implies $X \cong Z$, and a property $\script{P}$ of topological spaces is reconstructible if $\singletonDeletion{X}=\singletonDeletion{Z}$ implies ``$X$ has $\script{P}$ if and only if $Z$ has $\script{P}$".

It is shown in \cite{recpaper} that all the aforementioned examples are reconstructible, with the exception of the Cantor set where we have $\singletonDeletion{C}=\singletonDeletion{C\setminus \singleton{0}}$. This example also shows that compactness is a non-reconstructible property. These observations give rise to the question what underlying principles make a topological space or a topological property reconstructible.

\begin{recproblem}
Determine which topological spaces and properties are reconstructible and which ones fail to be reconstructible.
\end{recproblem}

The purpose of this paper is to describe two surprising results about topological reconstruction. Our first result shows that it is undecidable in ZFC whether $\wstar$, the Stone-\v{C}ech remainder of the integers, is reconstructible. Our second result is concerned with the question whether \emph{normality}, one of the fundamental topological separation properties, is reconstructible. We show that $\wstar$ provides a natural example establishing that the answer is consistently negative. 

The result about reconstruction of normality is somewhat curious, especially so because all other topological separation axioms are easily seen to be reconstructible for spaces containing at least three points \cite[Thm.~3.1]{recpaper}. In particular, the property of being a $T_1$, a Hausdorff or a Tychonoff space is reconstructible. But for normality, the best we could do in \cite{recpaper} was to show that normality of a space is reconstructible provided the space has at least one normal card. On the positive side, this result applies to every normal space containing a $G_\delta$-point, since normality is hereditary with respect to $F_\sigma$ subspaces \cite[2.1.E]{Eng}. In particular, normality is reconstructible in the realm of first-countable spaces. However, for an uncountable cardinal $\kappa$, the Cantor cube $2^\kappa$ provides an example of a compact Hausdorff space where all cards are non-normal. And given that the Cantor set $C \cong 2^\w$ is non-reconstructible, one might suspect that these large Cantor cubes are examples witnessing that normality is non-reconstructible. However, $2^\kappa$ is reconstructible if $\kappa$ is uncountable \cite[2.1]{recpaper}.

This is where the Stone-\v{C}ech remainder of the integers $\wstar$ enters the stage. It is a well-known open problem whether all cards of $\wstar$ are non-normal (cf.\ \cite[Q13]{openproblems} and \cite{szymanski}). Under the Continuum Hypothesis (CH), however, it is a classical theorem independently due to Rajagopalan \cite{raj} and Warren \cite{Warren} that the answer is yes. Thus, under CH, the space $\wstar$ is a further candidate to witness that normality might be non-reconstructible. 

And indeed, we show below that $\wstar$ provides a consistent example that normality is non-reconstructible. In our main result we prove that assuming CH, the space $\wstar \setminus \singleton{p}$ for a $P$-point $p$ of $\wstar$ (under CH, this is a point $p$ with a nested neighbourhood base) is a non-normal reconstruction of the normal space $\wstar$. Thus, CH implies that normality is non-reconstructible. We also show that CH (which is equivalent to $\w_1=\w_1^{<\w_1}$) can be weakened to the assumption that there is an uncountable cardinal $\kappa$ with $\kappa = \kappa^{<\kappa}$. We will see that under this cardinal assumption, the spaces $S_\kappa$ defined by Negrepontis in \cite{negrepontis} provide further examples that normality is non-reconstructible. This ties in well with the previous non-reconstruction results, as the spaces $S_\kappa$ can be thought of simultaneously generalising the behaviour of the Cantor set $C$ and of $\wstar$ under CH. 

This paper is organised as follows. In Section \ref{section3}, we recall the relevant definitions and facts about the spaces $\wstar$ and $S_\kappa$. In Section \ref{sectiondeck} we briefly describe the deck of $\wstar$, and then turn to a closer investigation of cards of $\wstar$ and $S_\kappa$ in Section \ref{sectionfincomp}. Our main technical result states that under CH, all finite compactifications of $\wstar \setminus \singleton{x}$ are homeomorphic to $\wstar$ and, moreover, that all but at most one point added at infinity will be $P$-points. 

In Section \ref{sectionnormality}, we recall why cards of $\wstar$ fail to be normal under CH. We then prove that for uncountable $\kappa$ with $\kappa = \kappa^{<\kappa}$, all cards of $S_\kappa$ are non-normal. In Section \ref{section4}, we prove the announced reconstruction results. We show that under CH, the space $\wstar$ is non-reconstructible, and more generally that under $\kappa = \kappa^{<\kappa}$, the space $S_\kappa$ is non-reconstructible. We also show that it follows from a theorem by van Douwen, Kunen and van Mill that it is consistent with MA$+\neg$CH that $\wstar$ is reconstructible. Finally, in Section \ref{sectionalltogether} we combine our results from the earlier sections to show that the existence of an uncountable cardinal $\kappa$ with $\kappa = \kappa^{<\kappa}$ implies that normality is non-reconstructible. We conclude the paper in Section \ref{sectionquestions} with some questions.

I would like to thank my advisor Rolf Suabedissen for inspiring discussions about the topological reconstruction problem.

\section{\texorpdfstring{What we need to know about $\wstar$ and $S_\kappa$}{Introduction to the Stone-Cech compactification of the integers}}
\label{section3}

In this section we recall crucial properties of the spaces $\wstar$ and $S_\kappa$ which we use in the course of this paper. All this and more can be found in \cite{Ultrafilters,Rings, Intro}. 

A subset $Y$ of a space $X$ is $\Cstar$-\emph{embedded} if every bounded real-valued continuous function on $Y$ can be extended to a continuous function on all of $X$. For a Tychonoff space $X$, we write $\beta X$ for its \emph{Stone-\v{C}ech compactification}, the unique compact Hausdorff space in which $X$ is densely $\Cstar$-embedded, and we write $X^*=\beta X \setminus X$ for the \emph{remainder} of $X$. A space is \emph{zero-dimensional} if it has a base of clopen (closed-and-open) sets. 

A subset of a Tychonoff space of the form $f^{-1}(0)$ for some real-valued continuous function $f$ is called a \emph{zero-set}. A \emph{cozero-set} is the complement of a zero-set. A Tychonoff space $X$ is an $F$-\emph{space} if each cozero-set is $\Cstar$-embedded in $X$. A $G_\delta$ subset of $X$ is a subset which can be expressed as a countable intersection of open sets, and an $F_\sigma$ subset is a complement of a $G_\delta$, namely a set which can be expressed as a countable union of closed subsets. We shall need the following facts about $F$-spaces \cite[Ch.\,14]{Rings}.

\begin{enumerate}
%\item $X$ is an $F$-space if and only if $\beta X$ is an $F$-space.
\item A normal space is an $F$-space if and only if disjoint open $F_\sigma$-sets have disjoint closures.
\item Closed subspaces of normal $F$-spaces are again $F$-spaces.
\item Infinite closed subspaces of compact $F$-spaces contain a copy of $\beta \w$. Therefore, compact $F$-spaces do not contain convergent sequences.
\end{enumerate}

The space $\wstar$ is a compact zero-dimensional Hausdorff space of weight $\cont=2^{\aleph_0}$ without isolated points with the following extra properties: it is an $F$-space in which each non-empty $G_\delta$-set has non-empty interior. A space with these properties is also called \emph{Parovi\v{c}enko space}. 

\begin{mythm}[Parovi\v{c}enko {\cite{parov}}, van Douwen and van Mill {\cite{douwenmill}}]
\label{parovtheorem}
CH is equivalent to the assertion that every Parovi\v{c}enko space is homeomorphic to $\wstar$.
\end{mythm}

A $P$\emph{-point} is a point $p$ such that any countable intersection of neighbourhoods of $p$ contains again a neighbourhood of $p$. In other words, $p$ is a $P$-point if $p$ is in the interior of every $G_\delta$ containing $p$. The existence of $P$-points in $\wstar$ was first proved as a consequence of the Continuum Hypothesis in \cite{Rudin}. The existence of $P$-points can also be shown under $MA+\neg CH$ \cite[2.5.5]{Intro}. On the other hand, it is consistent that $P$-points in $\wstar$ do not exist \cite[2.7]{Intro}.  We list some facts about $P$-points in $\wstar$ under CH.
\begin{enumerate}
\setcounter{enumi}{3}
\item Assuming [CH], a point $p \in \wstar$ is a $P$-point if and only if $p$ has a nested neighbourhood base if and only if $p$ is not contained in the boundary of any open $F_\sigma$-set.
\item Assuming [CH], for every pair of $P$-points in $\wstar$ there exists an autohomeomorphism of $\wstar$ mapping one $P$-point to the other \cite{Rudin}.
\end{enumerate}

We now describe the spaces $S_\kappa$, which generalise the behaviour of $\wstar$ under CH to higher cardinals $\kappa$. Note that in compact zero-dimensional spaces, cozero-sets are countable unions of clopen sets. This motivates the following definition from \cite[Ch.\,14]{Ultrafilters}. In a zero-dimensional space $X$, the ($X$-)\emph{type} of an open subset $U$ of $X$ is the least cardinal number $\tau=\tau(U)$ such that $U$ can be written as a union of $\tau$-many clopen subsets of $X$. A zero-dimensional space where open subsets of type less than $\kappa$ are $\Cstar$-embedded is called an $F_\kappa$-\emph{space}. In zero-dimensional compact spaces the notions of $F$- and $F_{\w_1}$-space coincide.

\begin{enumerate}
%\item[($1'$)] $X$ is a strongly zero-dimensional $F_\kappa$-space if and only if $\beta X$ is a strongly zero-dimensional $F_\kappa$-space. See Theorem \ref{FKappaSpaces}.
\item[($1'$)] A normal space is an $F_\kappa$-space if and only if disjoint open sets of types less than $\kappa$ have disjoint closures \cite[6.5]{Ultrafilters}.
\end{enumerate}

Following \cite[1.2]{Dow}, we call a space a $\kappa$-\emph{Parovi\v{c}enko space} if it is a compact zero-dimensional $F_\kappa$-space of weight $\kappa^{<\kappa}$ without isolated points such that non-empty intersections of fewer than $\kappa$ many open sets have non-empty interior. The space $\wstar$ is an ($\w_1$-)Parovi\v{c}enko space. %The notation $\kappa^{<\kappa}$ means $\sum_{\lambda < \kappa} \kappa^\lambda$.

\begin{mythm}[Negrepontis {\cite{negrepontis}}, Dow {\cite{Dow}}]
\label{negre}
The assumption $\kappa = \kappa^{<\kappa}$ is equivalent to the assertion that all $\kappa$-Parovi\v{c}enko spaces are homeomorphic.
\end{mythm}

If the condition $\kappa=\kappa^{<\kappa}$ is satisfied then the unique $\kappa$-Parovi\v{c}enko space exists and is denoted by $S_\kappa$ \cite[6.12]{Ultrafilters}. The space $S_\w$ is homeomorphic to the Cantor space, and whenever the space $S_{\w_1}$ exists, it is homeomorphic to $\wstar$. The existence of uncountable cardinals satisfying the equality $\kappa=\kappa^{<\kappa}$ is independent of ZFC but an assumption like $\kappa^+=2^\kappa$ implies the equality for $\kappa^+$. 

A $P_\kappa$\emph{-point} $p$ is a point such that the intersection of less than $\kappa$-many neighbourhoods of $p$ contains again an open neighbourhood of $p$. Thus, a  $P_{\w_1}$-point is a $P$-point. For a proof that $S_\kappa$ contains $P_\kappa$-points see \cite[6.17]{Ultrafilters} or \cite[4.6]{comppaper}. 
\begin{enumerate}
\item [($4'$)] Assume $\kappa=\kappa^{<\kappa}$. In $S_\kappa$, a point $p$ is a $P_\kappa$-point if and only if $p$ has a nested neighbourhood base if and only if $p$ is not contained in the boundary of any open set of type less than $\kappa$.  
\item [($5'$)] Assume $\kappa=\kappa^{<\kappa}$. For every pair of $P_\kappa$-points in $S_\kappa$ there exists an autohomeomorphism of $S_\kappa$ mapping one $P_\kappa$-point to the other \cite[6.21]{Ultrafilters}.
\end{enumerate}

\section{\texorpdfstring{The deck of $\wstar$}{The deck of w*}}
\label{sectiondeck}

We briefly describe the cards of $\wstar$. Recall that the cards of a space correspond to the non-homeomorphic subspaces one can obtain by deleting singletons. It turns out that cards of compact Hausdorff spaces, and in particular cards of $\wstar$ correspond to the different orbits under the action of its autohomeomorphism group. 

A space X is \emph{homogeneous} if for every pair of points $x$ and $y$ of X there exists a homeomorphism of $X$ carrying $x$ to $y$. In general, we say $x$ and $y$ lie in the same \emph{orbit} of $X$ if $x$ can be mapped to $y$ by a homeomorphism of $X$. The orbits form equivalence classes and the collection of orbits is denoted by $X/_\sim$. 

If $x$ and $y$ lie in the same orbit of $X$ then deleting either $x$ or $y$ gives the same card. For example, fact $(4)$ from the previous section yields that under CH, all cards of $\wstar$ obtained by deleting a $P$-point are homeomorphic. Also, the deck of a homogeneous space consists of only one card, and $\cardinality{\singletonDeletion{X}} \leq \cardinality{X/_{\sim}}$ is true for any space $X$. We now show that for compact Hausdorff spaces, we have equality in the previous line.

\begin{mythm}
\label{homogeneous}
Let $X$ be a compact Hausdorff space. Then $\cardinality{\singletonDeletion{X}} = \cardinality{X/_{\sim}}$ and the different cards correspond bijectively to the orbits of $X$.
\end{mythm}

\begin{proof}
As $\cardinality{\singletonDeletion{X}} \leq \cardinality{X/_{\sim}}$ by the remark preceding the theorem, it is enough to prove that $\cardinality{\singletonDeletion{X}} \geq \cardinality{X/_{\sim}}$. If $X$ is homogeneous, the result is clear. If $X$ is not homogeneous, find $x$ and $y$ contained in different orbits. Suppose for a contradiction that the two cards obtained by deleting $x$ and $y$ are homeomorphic, i.e.\ that there exists a homeomorphism $f \colon X \setminus \singleton{x} \to X \setminus \singleton{y}$. If $x$ is isolated, then both $X \setminus \singleton{x}$ and $X \setminus \singleton{y}$ are compact and hence $y$ must be isolated, too. But then, both points $x$ and $y$ lie in the same orbit, a contradiction.

Thus, we may assume that both $x$ and $y$ are non-isolated. Then $X$ is a one-point compactification of both $X \setminus \singleton{x}$ and $X \setminus \singleton{y}$. But  since all one-point compactifications of a locally compact space are homeomorphic by a map carrying remainders onto remainders \cite[3.5.11]{Eng}, the map 
$f \cup \Set{\langle x,y \rangle}$ is a homeomorphism of $X$. Thus, $x$ and $y$ are contained in the same orbit, a contradiction.
\end{proof}

Z. Frol\'ik proved (in ZFC) that every orbit in $\wstar$ is of size $\cont$, and therefore that $\wstar$ has $2^\cont$-many orbits \cite{Frolik}. It follows that $\cardinality{\singletonDeletion{\wstar}}= 2^\cont = \cardinality{\wstar}$, i.e.\ the space $\wstar$ has the maximal possible number of different cards.

\section{\texorpdfstring{Finite compactifications of cards of $\wstar$ and $S_\kappa$}{Finite compactifications}}
\label{sectionfincomp}

This section contains the technical groundwork for our non-reconstruction results. The main result is a characterisation of finite compactifications of cards of $\wstar$ and $S_\kappa$.

We call a point $x$ of a Hausdorff space $X$ a \emph{strong butterfly point} if its complement $X \setminus \singleton{x}$ can be partitioned into open sets $A$ and $B$ such that $\closure{A} \cap \closure{B} = \singleton{x}$. The sets $A$ and $B$ are called \emph{wings} of the butterfly point $x$. In the following we recall a theorem by Fine and Gillman \cite{fine,Gillmann} stating that under [CH], every point in $\wstar$ is a strong butterfly point. 

The result plays a crucial role in our proofs of Theorem \ref{ClassificationCompactifications} on the classification of finite compactifications of cards of $\wstar$ and in our proof of Theorem \ref{nonnormalcards}, where we show that cards of $\wstar$ are non-normal under CH.

\begin{mythm}[Butterfly Lemma, Fine and Gillman]
\label{butterfly}
\textnormal{[CH].} Every point in $\wstar$ is a strong butterfly point. \qed
\end{mythm}

\begin{mycor}
\label{nonStone}
\textnormal{[CH].} Every card of $\wstar$ has a two-point compactification. In particular, no card $\wstar \setminus \singleton{x}$ is $\Cstar$-embedded in $\wstar$.
\end{mycor}

\begin{proof}
It follows from the Butterfly Lemma that for every point $x$ in $\wstar$, we have $\wstar \setminus \singleton{x} = A \oplus B$ where $A$ and $B$ are non-compact open subsets of $\wstar$. Considering their respective one-point compactifications $\alpha A$ and $\alpha B$, we see that $\alpha A \oplus \alpha B$ is a two-point compactification of $\wstar \setminus \singleton{x}$. 

It follows immediately that the Stone-\v{C}ech compactification of $\wstar \setminus \singleton{x}$ cannot coincide with its one-point compactification. Therefore, $\wstar \setminus \singleton{x}$ is not $\Cstar$-embedded in $\wstar$.
\end{proof}

Analogues of these results for $S_\kappa$ were observed by Negrepontis \cite[14.2]{Ultrafilters}.

\begin{mythm}
\label{butterfly2}
Assume $\kappa = \kappa^{<\kappa}$. Every point in $S_\kappa$ is a strong butterfly point. \qed
\end{mythm}

\begin{mycor}
\label{nonStone2}
Assume $\kappa = \kappa^{<\kappa}$. Every card of $S_\kappa$ has a two-point compactification. In particular, no card $S_\kappa \setminus \singleton{x}$ is $\Cstar$-embedded in $S_\kappa$. \qed
\end{mycor}

We have established that cards of $\wstar$ under CH, and cards of $S_\kappa$ under $\kappa = \kappa^{<\kappa}$ have non-trivial finite compactifications. We now give a precise characterisation how finite compactifications of cards of these spaces look like. Our aim is to prove the following pair of theorems.

\begin{mythm}
\label{ClassificationCompactifications}
\textnormal{[CH].} For every $x \in \wstar$ we have that
\begin{enumerate}[(a)]
\item $\wstar \setminus \singleton{x}$ has arbitrarily large finite compactifications,
\item every finite compactification of $\wstar \setminus \singleton{x}$ is homeomorphic to $\wstar$, and
\item for every finite compactification, all but at most one point at infinity are $P$-points.
\end{enumerate}
\end{mythm}

The case for $S_\kappa$ looks exactly the same. 
\begin{mythm}
\label{ClassificationCompactificationsSkappa}
Assume $\kappa = \kappa^{<\kappa}$. For every $x \in S_\kappa$ we have that
\begin{enumerate}[(a)]
\item $S_\kappa \setminus \singleton{x}$ has arbitrarily large finite compactifications,
\item every finite compactification of $S_\kappa \setminus \singleton{x}$ is homeomorphic to $S_\kappa$, and
\item for every finite compactification, all but at most one point at infinity are $P_\kappa$-points.
\end{enumerate}
\end{mythm}

To prove these theorems, we begin with a sufficient condition for zero-dimensional locally compact Hausdorff spaces to have only one homeomorphism type amongst their finite compactifications.

\begin{mylem}
\label{theorem1999}
Let $X$ be a zero-dimensional compact Hausdorff space such that $X \oplus X$ is homeomorphic to $X$ and
\begin{enumerate}
\item[$(\star)$]  for every point $x$ of $X$, the one-point compactification of any clopen non-compact subset of $X \setminus \singleton{x}$ is homeomorphic to $X$.
\end{enumerate}
Then, for all $x$, all finite compactifications of $X \setminus \singleton{x}$ are homeomorphic to $X$.
\end{mylem}

\begin{proof}
Let $Z$ be a finite compactification of $X \setminus \singleton{x}$ with remainder consisting of points $\infty_1,\ldots, \infty_n$. By \cite[2.3]{Woods}, every finite compactification of a locally compact zero-dimensional space is zero-dimensional. Hence, there is a partition of $Z$ into $n$ disjoint clopen sets $A_i$ such that $\infty_i \in A_i$. 

The set $A_i \setminus \singleton{\infty_i}$ is a clopen non-compact subspace of $X \setminus \singleton{x}$. Therefore, by property ($\star$) and uniqueness of the one-point compactification, it follows that $A_i$ is homeomorphic to $X$. This proves, after applying $X \oplus X \cong X$ iteratively, that $Z$ is homeomorphic to $X$.  
\end{proof}

This lemma lies at the heart of our proofs for Theorems  \ref{ClassificationCompactifications} and \ref{ClassificationCompactificationsSkappa}. Let us see that it applies to both $\wstar$ under CH, and to $S_\kappa$ assuming $\kappa = \kappa^{<\kappa}$.

\begin{mylem}[{\cite[3.4]{comppaper}}]
\label{lemma123}
\textnormal{[CH].} The space $\wstar$ has property $(\star)$, i.e.\ for every $x$ the one-point compactification of a clopen non-compact subset of $\wstar \setminus \singleton{x}$ is homeomorphic to $\wstar$.
\end{mylem}

\begin{proof}
Let $A$ be a clopen non-compact subset of $\wstar \setminus \singleton{x}$. Taking $A \cup \singleton{x}$, a closed subset of $\wstar$, as representative of its one-point compactification, we use fact $(2)$ from Section \ref{section3} to see that it is a zero-dimensional compact $F$-space of weight $\cont$ without isolated points. 

Suppose that $U \subset A \cup \singleton{x}$ is a non-empty $G_\delta$-set. If $U$ has empty intersection with $A$, then the singleton $U=\singleton{x}$ is a $G_\delta$-point, and hence has countable character in the compact Hausdorff space $A \cup \singleton{x}$ \cite[3.3.4]{Eng}. It follows that there is a non-trivial sequence in $\wstar$ converging to $x$, contradicting fact $(3)$ from Section \ref{section3}. Thus, $U$ intersects the open set $A$ and their intersection is a non-empty $G_\delta$-set of $\wstar$ with non-empty interior.

Applying Parovi\v{c}enko's Theorem \ref{parovtheorem} completes the proof.
\end{proof}

For a proof of the result in case of the space $S_\kappa$, we refer the reader to our earlier paper \cite{comppaper}.

\begin{mylem}[{\cite[4.4]{comppaper}}]
\label{lemma123b}
Assume $\kappa = \kappa^{<\kappa}$. The space $S_\kappa$ has property $(\star)$, i.e.\ for every $x$ the one-point compactification of a clopen non-compact subset of $S_\kappa \setminus \singleton{x}$ is homeomorphic to $S_\kappa$. \qed
\end{mylem}

Thus, we have verified that Lemma \ref{theorem1999} does apply to our spaces $\wstar$ and $S_\kappa$. Surprisingly, despite its strong assumptions, the lemma itself applies to a variety of interesting spaces. 

Spaces which only have $\lambda$ different homeomorphism types amongst their open subspaces (for some cardinal $\lambda$) are said to be of \emph{diversity} $\lambda$ \cite{diversity}. One checks that Lemma \ref{theorem1999} applies to all compact Hausdorff spaces of diversity two, which are known to be zero-dimensional \cite{zerodiv}. In particular, Lemma \ref{theorem1999} applies to the Cantor space $C$, which can be characterised as the unique compact metrizable space of diversity two \cite{Cantor}. It also applies to the Alexandroff Double Arrow space $D$ and to the product $D \times C$. Incidentally, these examples are also non-reconstructible \cite[2.5]{recpaper}.

In a compact Hausdorff space $X$ of diversity two, any subspace $X \setminus \singleton{x}$ is homeomorphic to $X \setminus \Set{x_1,\ldots,x_n}$ and therefore has arbitrarily large finite compactifications. This is when Lemma \ref{theorem1999} is most valuable. Our next lemma shows that not much is needed for this scenario to occur. The proof is a simple induction. 
  
\begin{mylem}
\label{lemm2}
Let $X$ be a topological space such that for all $x$, all finite compactifications of $X \setminus \singleton{x}$ are homeomorphic to $X$. If all spaces $X \setminus \singleton{x}$ have two-point compactifications, they have arbitrarily large finite compactifications. \qed
\end{mylem}

The following example of the Cantor cube $2^\kappa$ for uncountable $\kappa$ shows that the assumptions in Lemma \ref{lemm2} cannot be considerably weakened. Since $\beta (2^\kappa \setminus \singleton{x})=2^\kappa$ \cite[Thm.\ 2]{Glicksberg}, these spaces have a unique compactification. The cube $2^\kappa$ is a zero-dimensional compact Hausdorff space with $2^\kappa \cong 2^\kappa  \oplus 2^\kappa$. For property $(\star)$, let $A \subset 2^\kappa \setminus \singleton{x}$ be a clopen non-compact subset. Since $2^\kappa \setminus \singleton{x}$ does not have a 2-point compactification, $A \cup \singleton{x}$ must be clopen in $2^\kappa$. But every clopen set of $2^\kappa$ can be written as a disjoint union of finitely many product-basic open sets, which are homeomorphic to $2^\kappa$. Hence $A \cup \singleton{x} \cong 2^\kappa$. We conclude that Lemma \ref{theorem1999} applies, but restricts to the obvious assertion that the one-point compactification of $2^\kappa \setminus \singleton{x}$ is homeomorphic to $2^\kappa$.

Now finally, with these lemmas established, it is not difficult to give proofs of the main theorems in this section.

\begin{proof}[Proof of Theorem \ref{ClassificationCompactifications}]
Assertion $(b)$ is an immediate consequence of Lemmas \ref{theorem1999} and \ref{lemma123}. For $(a)$, note that Corollary \ref{nonStone} implies that every card of $\wstar$ has a two-point compactification, which is homeomorphic to $\wstar$ by $(b)$. Therefore, $(a)$ now follows from Lemma \ref{lemm2}.

For assertion $(c)$, suppose there is a finite compactification $Z$ of $\wstar \setminus \singleton{x}$ containing two non-$P$-points $\infty_1$ and $\infty_2$ at infinity. Then there are disjoint open $F_\sigma$-sets $F_1$ and $F_2$ in $Z$ with $F_i \subset \wstar \setminus \singleton{x}$ containing $\infty_1$ and $\infty_2$ in their respective boundaries. However, in $\wstar$ the disjoint non-compact open $F_\sigma$-sets $F_1$ and $F_2$ both limit onto $x$. This contradicts the $F$-space property of $\wstar$.
\end{proof} 

In ZFC, the above argument still shows that every finite compactification of $\wstar \setminus \singleton{x}$ is a Parovi\v{c}enko space of weight $\cont$ such that at most one point at infinity is not a $P$-point. However, one cannot decide in ZFC alone whether there are finite compactifications of $\wstar \setminus \singleton{x}$ other than the one-point compactification \cite{douwenkunenmill}.

\begin{proof}[Proof of Theorem \ref{ClassificationCompactificationsSkappa}]
Assertions $(a)$ and $(b)$ follow as in the previous proof.

The proof of $(c)$ uses the same idea as in the case of $\wstar$. Suppose there is a finite compactification $Z$ of $S_\kappa \setminus \singleton{x}$ containing two non-$P_\kappa$-points $\infty_1$ and $\infty_2$ at infinity. Then there are disjoint open subsets $F_1$ and $F_2$ in $Z$ of type less than $\kappa$ with $F_i \subset S_\kappa \setminus \singleton{x}$ that contain $\infty_1$ and $\infty_2$ in their respective boundaries. However, in $S_\kappa$ the disjoint non-compact open sets $F_1$ and $F_2$ of type less than $\kappa$ both  limit onto $x$, contradicting the $F_\kappa$-space property.
\end{proof}

\section{\texorpdfstring{Non-normality of cards of $\wstar$ and $S_\kappa$}{Non-normality of cards}}
\label{sectionnormality}

This section contains proofs that under CH, every card of $\wstar$ is non-normal, and that for uncountable $\kappa$ with $\kappa = \kappa^{<\kappa}$, every card of $S_\kappa$ is non-normal. 

In case of $\wstar$, this result is originally due to Rajagopalan and Warren. Proofs can be found in \cite{easy, raj,Warren}. An account of \cite{raj} is contained in \cite[7.2-7.4]{Walker}. In the first part of this section, we advertise a different, very elegant proof of this classical theorem. The proof builds on ideas from \cite{logunov, sapirowskii, terasawa}, and is mentioned, without details, in \cite[Rmk.\ 3]{terasawa}. The result that also cards of $S_\kappa$ are non-normal is new. Our proof generalises the approach of \cite{Warren} for $\wstar$.

To begin, we note that normality of cards of the spaces $\beta \w$ and $\wstar$ are in fact equivalent problems, in the sense that when deleting a point of the remainder of $\beta \w$, the card of $\beta \w$ is normal if and only if the corresponding card of $\wstar$ is normal. This result is folklore, but as no proof is available in the standard literature we give one in the next lemma. 
 
\begin{mylem}
\label{equivalentnormality}
Let $x \in \wstar$. Then $\beta \w \setminus \singleton{x}$ is normal if and only if $\wstar \setminus \singleton{x}$ is normal.
\end{mylem}

\begin{proof}
If $\beta \w \setminus \singleton{x}$ is a normal space then its closed subspace $\wstar \setminus \singleton{x}$ inherits normality. 

For the converse implication, assume that $\wstar \setminus \singleton{x}$ is normal and let $C_1$ and $C_2$ be disjoint closed subsets of $\beta \w \setminus \singleton{x}$. Let $C'_i = C_i \cap \wstar$ be the part of each closed set lying in the remainder of $\wstar$. Since $\wstar \setminus \singleton{x}$ is normal, there are open subsets $W_1$ and $W_2$ of $\wstar \setminus \singleton{x}$ with disjoint closures such that $C'_i \subset W_i$. Note that all open sets $V_1$ and $V_2$ in $\beta \w \setminus \singleton{x}$ with $W_i = V_i \cap \wstar$ can only intersect in a finite subset of $\w \subset \beta \w$. Hence $U_1=V_1 \setminus V_2$ and $U_2=V_2 \setminus V_1$ are disjoint open subsets of $\beta \w \setminus \singleton{x}$ containing $C'_1$ and $C'_2$ respectively. 

But now, as any subset of $\w$ is open in $\beta \w$, the sets 
$$ (U_1 \setminus C_2) \union (C_1\cap \w) \; \text{ and } \; (U_2 \setminus C_1) \union (C_2\cap \w)$$
are disjoint open subsets of $\beta \w \setminus \singleton{x}$ containing $C_1$ and $C_2$ respectively. This completes the proof that $\beta \w \setminus \singleton{x}$ is normal.
\end{proof}

\begin{mythm}
\label{nonnormalcards}
\textnormal{[CH].} Every card $\wstar \setminus \singleton{x}$ of $\wstar$ is non-normal.
\end{mythm}

\begin{proof}
Let $x \in \wstar$. By Lemma \ref{equivalentnormality} it suffices to show that $\beta \w \setminus \singleton{x}$ is non-normal. 
Assume for a contradiction that $\beta \w \setminus \singleton{x}$ is normal. By Tietze's Theorem \cite[2.1.8]{Eng}, the closed subset $\wstar \setminus \singleton{x}$ is $\Cstar$-embedded in $\beta \w \setminus \singleton{x}$, which in turn is $\Cstar$-embedded in $\beta \w$ \cite[3.6.9]{Eng}.
 
Thus, $\wstar \setminus \singleton{x}$ is $\Cstar$-embedded in $\wstar$, contradicting Corollary \ref{nonStone}. 
\end{proof}

We now come to the result that cards of $S_\kappa$ are non-normal. We will see that Lemma \ref{lemma123b} implies that without loss of generality we may focus our attention on cards obtained by deleting a $P_\kappa$-point of $S_\kappa$. From there on, we adapt Warren's result \cite[I.1]{Warren} that under CH, cards of $\wstar$ obtained by deleting $P$-points are non-normal.

\begin{mythm}
\label{nonnormalskappa}
Assume $\kappa = \kappa^{<\kappa}$. If $\kappa$ is uncountable then every card $S_\kappa \setminus \singleton{x}$ of $S_\kappa$ is non-normal.
\end{mythm}

\begin{proof}
It suffices to prove the theorem for cards that have been obtained by deleting a $P_\kappa$-point of $S_\kappa$. To see this, we use that by Theorem \ref{butterfly2}, every point $x \in S_\kappa$ is a butterfly point of $S_\kappa$ with wings $A$ and $B$. Using Theorem \ref{ClassificationCompactificationsSkappa}(c), we may assume that $x$ is a $P_\kappa$-point with respect to $A \cup \singleton{x}$, which in turn is homeomorphic to $S_\kappa$ by Lemma \ref{lemma123b}. Thus, if $S_\kappa$ removed a $P_\kappa$-point is non-normal, then $S_\kappa \setminus \singleton{x}$ contains the closed, non-normal subspace $A$, and hence is itself non-normal. 

So let $p$ be a $P_\kappa$-point in $S_\kappa$. We show that $S_\kappa \setminus \singleton{p}$ is non-normal. Fix a strictly decreasing neighbourhood base $\set{U_\alpha}:{\alpha < \kappa}$ of $p$ consisting of clopen sets. Pick $P_\kappa$-points $p_\alpha$ inside the non-empty sets $V_\alpha = U_\alpha \setminus U_{\alpha + 1}$, which is possible as $V_\alpha$ is homeomorphic to $S_\kappa$. Again, for each $p_\alpha \in V_\alpha$ we fix a nested neighbourhood base $\set{V_{\alpha,\beta}}:{\beta < \kappa}$ of clopen sets, such that $V_{\alpha,0}=V_\alpha$.  

We now describe two closed disjoint sets $A$ and $B$ of $S_\kappa \setminus \singleton{p}$ that cannot be separated by open sets, showing that this space is non-normal. Define, for each limit ordinal $\lambda<\kappa$, the sets
$$B_\lambda = \closure{\Union_{\alpha < \lambda} (V_\alpha \setminus V_{\alpha,\lambda})} \cap \bigcap_{\alpha < \lambda}U_\alpha.$$ 
The closures are of course taken in $S_\kappa \setminus \singleton{p}$. %Intersecting with $\bigcap_{\alpha < \lambda}U_\alpha$  guarantees that whenever $\lambda < \lambda'$ are limit ordinals less than $\kappa$, then $B_\lambda$ and $B_{\lambda'}$ are disjoint. 
Then put 
$$A=\closure{\set{p_\alpha}:{\alpha < \kappa}} \quad \text{and} \quad B = \Union_{\lambda < \kappa} B_\lambda.$$

Let us see why $A$ and $B$ are disjoint. Consider the disjoint sets $\Union_{\alpha < \lambda} V_{\alpha,\lambda}$ and $\Union_{\alpha < \lambda} V_\alpha \setminus V_{\alpha,\lambda}$, the first of which being a superset of $\set{p_\alpha}:{\alpha < \lambda}$. Both sets are of $S_\kappa$-type less than $\kappa$ and hence have disjoint closures in $S_\kappa$ by fact $(1')$ from Section \ref{section3}. It follows that $A$ and $B$ are disjoint, since if $q \in A \cap B$ then $q \notin U_\lambda$ for some limit ordinal $\lambda$, and we obtain a contradiction from 
$$q \in A \cap B \cap (S_\kappa \setminus U_\lambda) \subset \closure{\Union_{\alpha < \lambda} V_{\alpha,\lambda}}  \cap \closure{ \Union_{\alpha < \lambda} V_\alpha \setminus V_{\alpha,\lambda}} = \emptyset.$$

We now show that $B$ is closed. Suppose that $q$ lies in $\closure{B}$. Since $V_\alpha \cap B = \emptyset$ for all $\alpha < \kappa$ it follows that $q \in \bigcap_{\alpha < \mu} U_\alpha \setminus U_\mu$ for some limit ordinal $\mu < \kappa$. Note that $$B=\Union_{\lambda < \mu} B_\lambda \cup B_\mu \cup \Union_{\mu < \lambda < \kappa} B_\lambda.$$
Since the last factor is a subset of the closed set $U_\mu$, the point $q$ is contained in the closure of the first two factors. However, it follows from the construction that $\closure{\Union_{\lambda < \mu} B_\lambda} \cap \bigcap_{\alpha < \mu} U_\alpha \setminus U_\mu \subset B_\mu$. So $q$ lies in $\closure{B_\mu}=  B_\mu$. Hence $q \in B$, and we have shown that $B$ is closed.  

To complete the proof it remains to show that $A$ and $B$ cannot be separated by open sets. So let $U$ and $V$ be open sets of $S_\kappa \setminus \singleton{p}$ containing $A$ and $B$ respectively. For every ordinal $\alpha<\kappa$ there exists $\beta_\alpha > \alpha$ such that $V_{\alpha,{\beta_\alpha}} \subset  U$. Since $\kappa=\kappa^{<\kappa}$ implies that $\kappa$ is regular \cite[1.27]{Ultrafilters}, the increasing sequence defined by $\alpha_0=0$ and $\alpha_n=\beta_{\alpha_{n-1}}$ has a supremum $\gamma < \kappa$. Consider the set
$$W = \Union_{n \in \w} (V_{\alpha_n,\beta_{\alpha_n}} \setminus V_{\alpha_n,\gamma}).$$
It follows from our construction that 
$$W \subset U \quad \text{and} \quad W \subset \Union_{n \in \w} (V_{\alpha_n} \setminus V_{\alpha_n,\gamma}).$$
Let us see that $\closure{W} \cap \bigcap_{\alpha < \gamma}U_\alpha$ is a non-empty subset of $B_\gamma$. For this, we only have to show that $\closure{W}$ intersects $\bigcap_{\alpha < \gamma}U_\alpha =\bigcap_{n \in \w}U_{\alpha_n}$. But this holds, since otherwise the collection $\set{ S_\kappa \setminus U_{\alpha_n}}:{n \in \w}$ forms an open cover of the compact set $\closure{W}$, yielding a contradiction.

It follows that $V$ is a neighbourhood of every point in $\closure{W} \cap \bigcap_{\alpha < \gamma}U_\alpha$, and therefore that $V\cap W \neq \emptyset$. Since 
 $V \cap U \supseteq V \cap W$ we see that $U$ and $V$ cannot be disjoint, completing the proof.
\end{proof}

%----------------------------------- NEW CHAPTER -------------------------------------
\section{\texorpdfstring{Reconstruction results for $\wstar$ and $S_\kappa$}{Reconstruction of w* and Skappa}}
\label{section4}

This section contains our reconstruction results for the spaces $\wstar$ and $S_\kappa$. We will see that it is independent of ZFC whether the space $\wstar$ is reconstructible. For example, $\wstar$ is reconstructible in models where $\wstar$ is the Stone-\v{C}ech compactification of one of its cards, and is not reconstructible in models where the Continuum Hypothesis holds. 

Generalising the behaviour of $\wstar$ under CH, we show further below that assuming $\kappa = \kappa^{<\kappa}$, the spaces $S_\kappa$ are always non-reconstructible.

%We begin with the result that under CH, the space $\wstar$ is non-reconstructible.

\begin{mythm}
\label{mainresultsrec}
\textnormal{[CH].} The space $\wstar$ is non-reconstructible. For a $P$-point $p$, the space $\wstar \setminus \singleton{p}$ is a non-homeomorphic reconstruction of $\wstar$.
\end{mythm}

For the proof we need three lemmas describing the behaviour of quotients of Parovi\v{c}enko spaces $X$ when identifying a subset $A$ with a single point. Write $X/A$ for the quotient space induced by the partition $\Set{A} \union \set{\singleton{x}}:{x \in X \setminus A}$.

\begin{mylem} 
\label{quotient1}
Let $X$ be a compact Hausdorff space and $A \subset X$ a closed, non-open subset of $X$. Then $X / A$ is a one-point compactification of $X \setminus A$. Moreover, if $X$ is zero-dimensional, then so is $X/A$.
\end{mylem}
\begin{proof}
First, a quotient space of a compact space is compact. Further, $X/A$ is Hausdorff as $X$ is regular and $A$ is closed.
Since $A$ is not open, the map $X \setminus A \hookrightarrow X/A$, sending $x \mapsto \singleton{x}$ is a dense embedding with a one-point remainder. 

For zero-dimensionality, we show that $A \subset X$ has a neighbourhood base of clopen sets. So let $U$ be an open neighbourhood of $A$. Assuming that $X$ is zero-dimensional, for every $x \in A$ there is a clopen set such that  $x \in C(x) \subseteq U$. The clopen cover $\set{C(x)}:{x \in A}$ of the compact set $A$ has a finite subcover. Its union is a clopen set between $A$ and $U$. 
\end{proof}

The next lemma shows in which cases quotients preserve the $G_\delta$-property of Parovi\v{c}enko spaces. A similar lemma appears without proof in \cite[1.4.2]{Intro}. Note that if $X$ is compact Hausdorff and $A \subset X$ is closed, the quotient map $\pi \colon X \to X/A$ is a continuous map from a compact space to a Hausdorff space, and therefore closed. 

\begin{mylem} 
\label{quotient2}
Suppose $X$ has the property that non-empty $G_\delta$-sets have non-empty interior. Let $A \subset X$ be a closed, nowhere dense subset of $X$. Then $X / A$ also has the property that non-empty $G_\delta$-sets have non-empty interior. 
\end{mylem}
\begin{proof}
Let $U$ be a non-empty $G_\delta$ of $X/A$. We have to show that it has non-empty interior. Since $\pi$ is continuous and surjective, $\pi^{-1}(U)$ is a non-empty $G_\delta$-set of $X$. By assumption, it has non-empty interior. Since $A$ is closed and nowhere dense, the set $\pi^{-1}(U) \setminus A$ also has non-empty interior. Observing that $\pi(\interior{\pi^{-1}(U) \setminus A})$ is an open subset of $U$ completes the proof.
\end{proof}

Our last lemma tells us under which conditions collapsing a subset to a single point leaves the $F$-space property intact. The result is a slight generalisation of \cite[1.4.1]{Intro}. 

\begin{mylem}
\label{quotient3}
Let $X$ be a compact $F$-space and $A\subset X$ a closed subset containing at most one non-$P$-point of $X$. Then $X/A$ is an $F$-space.
\end{mylem}

\begin{proof}
By Lemma \ref{quotient1}, the space $X/A$ is normal. To establish the $F$-space property, it therefore suffices, using fact $(1)$ from Section \ref{section3}, to verify that disjoint open $F_\sigma$-sets have disjoint closures.

So let $U$ and $V$ be disjoint open $F_\sigma$-sets of $X/A$. Since $\pi^{-1}(U)$ and $\pi^{-1}(V)$ are disjoint open $F_\sigma$-sets of $X$, they have disjoint closures in $X$.

Suppose that $\Set{A} \in U \cup V$. Without loss of generality, we have $A \subset \pi^{-1}(U)$ and hence $\pi(\closure{\pi^{-1}(U)}) \cap \pi(\closure{\pi^{-1}(V)}) = \emptyset$. Since $\pi$ is a closed surjective map, we have $\closure{U} \subseteq \pi(\closure{\pi^{-1}(U)})$ and hence $\closure{U} \cap \closure{V} = \emptyset$. 

Now suppose that $\Set{A} \notin U \union V$. Then $A$ does not intersect $\pi^{-1}(U) \cup \pi^{-1}(V)$. Note that for all $P$-points $p \in A$ it follows from fact $(4)$ in Section \ref{section3} that 
$$p \notin \closure{\pi^{-1}(U)} \cup \closure{\pi^{-1}(V)}.$$
 Finally, since $\pi^{-1}(U)$ and $\pi^{-1}(V)$ have disjoint closures in $X$, the single non-$P$-point of $A$ cannot be contained in both of them. Thus, we may assume without loss of generality that $\Set{A} \notin \pi(\closure{\pi^{-1}(U)})$. Therefore, $\pi(\closure{\pi^{-1}(U)}) \cap \pi(\closure{\pi^{-1}(V)}) = \emptyset$, implying, as before, that $U$ and $V$ have disjoint closures in $X/A$. 
\end{proof}

\begin{proof}[Proof of Theorem \ref{mainresultsrec}]
We prove that for a $P$-point $p$ of $\wstar$, the space $\wstar \setminus \singleton{p}$ is a non-homeomorphic reconstruction of $\wstar$. Let us first show $\singletonDeletion{\wstar \setminus \singleton{p}} \subseteq \singletonDeletion{\wstar}$. 

For this inclusion, pick any card $\wstar \setminus \Set{p,x}$ in $\singletonDeletion{\wstar \setminus \singleton{p}}$. We claim that its one-point compactification $X=\singleton{\infty} \cup \p{\wstar\setminus \Set{p,x}}$ is a Parovi\v{c}enko space. Then, by Theorem \ref{parovtheorem}, there is a homeomorphism $f \colon X \to \wstar$. It follows 
$$\wstar \setminus \Set{p,x} \cong X \setminus \singleton{\infty} \cong \wstar \setminus \Set{f(\infty)},$$ 
establishing that $\wstar \setminus \Set{p,x} \in \singletonDeletion{\wstar}$.

To see that $X$ is a Parovi\v{c}enko space note that $X$ is a compact space of weight $\cont$ without isolated points. By Lemma \ref{quotient1}, we may take $\wstar/A$ with $A = \Set{p,x}$ as a representative for $X$, showing that $X$ is zero-dimensional. Further, by Lemmas \ref{quotient2} and \ref{quotient3}, the space $\wstar/A$ is an $F$-space with the property that non-empty $G_\delta$-sets have non-empty interior. Thus, $X$ is Parovi\v{c}enko, completing the proof of the first inclusion.

We now establish the reverse inclusion $\singletonDeletion{\wstar \setminus \singleton{p}} \supseteq \singletonDeletion{\wstar}$. For this, let $\wstar \setminus \singleton{x}$ be any card in $\singletonDeletion{\wstar}$. It follows from Theorem \ref{ClassificationCompactifications} that there exist points $\infty_1$ and $\infty_2$ of $\wstar$, of which $\infty_1$ is a $P$-point, such that 
$$\wstar \setminus \singleton{x} \cong \wstar \setminus \Set{\infty_1,\infty_2}.$$
By fact $(5)$ from Section \ref{section3}, there exists a homeomorphism $f$ of $\wstar$ carrying $\infty_1$ to $p$. Then $\wstar \setminus \singleton{x} \cong \wstar \setminus \Set{p,f(\infty_2)}$, and hence $\wstar \setminus \singleton{x}$ is a card in $\singletonDeletion{\wstar \setminus \singleton{p}}$.
\end{proof}

The proof that $S_\kappa$ is non-reconstructible is very similar. The following two lemmas are straightforward adaptions of Lemma \ref{quotient2} and \ref{quotient3} respectively.

\begin{mylem} 
\label{quotient4}
Suppose $X$ has the property that every non-empty intersection of fewer than $\kappa$ many open sets has non-empty interior. Let $A \subset X$ be a closed, nowhere dense subset of $X$. Then $X / A$ also has the property that every non-empty intersection of fewer than $\kappa$ many open sets has non-empty interior. \qed
\end{mylem}

\begin{mylem}
\label{quotient5}
Let $X$ be a compact $F_\kappa$-space and $A\subset X$ a closed subset containing at most one non-$P_\kappa$-point of $X$. Then $X/A$ is an $F_\kappa$-space. \qed
\end{mylem}

\begin{mythm}
\label{mainresultsrec2}
Assume $\kappa = \kappa^{<\kappa}$. The space $S_\kappa$ is non-reconstructible. Indeed, for a $P_\kappa$-point $p$, the space $S_\kappa \setminus \singleton{p}$ is a non-homeomorphic reconstruction of $S_\kappa$.
\end{mythm}

\begin{proof}
To prove the inclusion $\singletonDeletion{S_\kappa \setminus \singleton{p}} \subseteq \singletonDeletion{S_\kappa}$, pick any card $S_\kappa \setminus \Set{p,x}$ in $\singletonDeletion{S_\kappa \setminus \singleton{p}}$. Using Lemmas \ref{quotient4} and \ref{quotient5} we see that its one-point compactification $X=\singleton{\infty} \cup \p{S_\kappa\setminus \Set{p,x}}$ is a $\kappa$-Parovi\v{c}enko space. By Theorem \ref{negre} there is a homeomorphism $f \colon X \to S_\kappa$. It follows 
$$S_\kappa \setminus \Set{p,x} \cong X \setminus \singleton{\infty} \cong S_\kappa \setminus \Set{f(\infty)},$$ 
establishing that $S_\kappa \setminus \Set{p,x} \in \singletonDeletion{S_\kappa}$.

For the reverse inclusion, let $S_\kappa \setminus \singleton{x}$ be a card in $\singletonDeletion{S_\kappa}$. It follows from Theorem \ref{ClassificationCompactificationsSkappa} that there are points $\infty_1$ and $\infty_2$ of $S_\kappa$, with $\infty_1$ being a $P_\kappa$-point, such that 
$$S_\kappa \setminus \singleton{x} \cong S_\kappa \setminus \Set{\infty_1,\infty_2}.$$ 
By fact $(5')$ from Section \ref{section3}, there exists a homeomorphism $f$ of $S_\kappa$ carrying $\infty_1$ to $p$. Then $S_\kappa \setminus \singleton{x} \cong S_\kappa \setminus \Set{p,f(\infty_2)}$, and hence $S_\kappa \setminus \singleton{x}$ is a card in $\singletonDeletion{S_\kappa \setminus \singleton{p}}$.
\end{proof}

We conclude this section with the result that $\wstar$ is consistently reconstructible. In particular, together with the results above we see that the question whether $\wstar$ is reconstructible is independent of the axioms of set theory ZFC. Note also that for showing that $\wstar$ is non-reconstructible, the assumption CH cannot be weakened to Martin's axiom (MA).

\begin{mythm}[{\cite{douwenkunenmill}}]
\label{douwenkunenmilltheorem}
It is consistent with $MA+\cont=\w_2$ that for all $x\in \wstar$ we have $\beta(\wstar \setminus \singleton{x})=\wstar$. \qed
\end{mythm}
\begin{mythm}
It is consistent with $MA+\cont=\w_2$ that $\wstar$ is reconstructible.
\end{mythm}
\begin{proof}
It is shown in \cite[5.4]{recpaper} that every compact Hausdorff space arising as a non-trivial Stone-\v{C}ech compactification is reconstructible. Hence, the reconstruction result follows from the previous theorem.
\end{proof}

\section{Normality is consistently non-reconstructible}
\label{sectionalltogether}

Theorem \ref{mainresultsrec} established that under CH, the spaces $\wstar$ and $\wstar \setminus \singleton{p}$ for a $P$-point $p$ are non-homeomorphic reconstructions of each other. Since under CH, the space $\wstar \setminus \singleton{p}$ is non-normal by Theorem \ref{nonnormalcards}, this gives the desired result that normality is consistently non-reconstructible. Also, in presence of Hausdorffness, compactness implies paracompactness, which in turn implies normality \cite[5.1.1 \& 5.1.18]{Eng}. Hence, it is also consistent that paracompactness is non-reconstructible.

More generally, we have the following theorem.

\begin{mythm}
\label{thedefinitenormalityresult}
The existence of an uncountable cardinal $\kappa$ with the property $\kappa=\kappa^{<\kappa}$ implies that normality and paracompactness are not reconstructible. 
\end{mythm}

\begin{proof}
The space $S_\kappa$ is compact Hausdorff, hence paracompact and normal. By Theorem \ref{nonnormalskappa}, the space $S_\kappa \setminus \singleton{p}$ for a $P_\kappa$-point $p$ is non-normal and non-paracompact. However, by Theorem \ref{mainresultsrec2}, both spaces are reconstructions of each other. Therefore, the properties of being normal or paracompact are not reconstructible.
 \end{proof}

\section{Questions}
\label{sectionquestions}

In the previous section we have shown that normality and paracompactness are consistently non-reconstructible. Are these results true in ZFC?

\begin{myquest}
Are normality or paracompactness non-reconstructible properties?
\end{myquest} 

%When looking for a ZFC example showing that normality is non-reconstructible, note that by our discussion in the introduction, such an example has to be a normal space whose cards are all non-normal. Such spaces seem to be rare. In ZFC, the only such spaces we are aware of are related to normal products $ \prod_{\alpha < \kappa} X_\alpha$ where each $X_\alpha$ is a Hausdorff space containing at least two points, and $\kappa$ is an uncountable cardinal. However, the cubes $2^\kappa$ and $I^\kappa$ for uncountable $\kappa$ are reconstructible \cite[2.1]{recpaper}.
%
%\begin{myquest}
%Are there further nice $ZFC$-examples of normal spaces such that every card is non-normal---and are they reconstructible?
%\end{myquest}

Next, we have shown that it is consistent with the negation of CH that $\wstar$ is reconstructible. Is this always the case? Note that our present proof of the fact that $\wstar$ is non-reconstructible under CH uses, on several occasions, the full power of Parovi\v{c}enko's theorem---which itself is equivalent to CH.

\begin{myquest}
Is CH equivalent to the assertion that $\wstar$ is non-reconstructible?
\end{myquest}

Our last question asks whether in Theorem \ref{mainresultsrec}, we can tell which point precisely one has to delete from $\wstar \setminus \singleton{p}$ in order to obtain a given card of $\wstar$. For example, it is easy to see that under CH, if $q$ is a further $P$-point then $\wstar \setminus \Set{p,q} \cong \wstar \setminus \singleton{q}$. Does this behaviour occur for all points of $\wstar$? 
 
\begin{myquest}
\textnormal{[CH].} Let $p$ be a $P$-point of $\wstar$. Is it true that for all $x$ we have $\wstar \setminus \Set{p,x} \cong \wstar \setminus \singleton{x}$?
\end{myquest}

%Finally, we have seen in Theorem \ref{mainresultsrec} that under CH, the card $\wstar \setminus \singleton{p}$ for a $P$-point $p$ of $\wstar$ is not reconstructible. What about the other cards in $\singletonDeletion{\wstar}$? 
%
%\begin{myquest}
%Are cards of $\wstar$ obtained by deleting non-$P$-points reconstructible?
%\end{myquest}
%
%Note that these cards are not reconstructions of the space $\wstar$. Indeed, if $x_1$ and $x_2$ are non-$P$-points of $\wstar$, then $\wstar \setminus \singleton{x_1,x_2} \notin \singletonDeletion{\wstar}$. Otherwise, if $\wstar \setminus \singleton{x_1,x_2}$ is homeomorphic to some $\wstar \setminus \singleton{x}$, then this card $\wstar \setminus \singleton{x}$ has a two-point compactification with two non-$P$-points $x_1$ and $x_2$ in its remainder, contradicting Theorem \ref{ClassificationCompactifications}.

%--- BIBLIOGRAPHY---
	  
\end{document}